\documentclass[reqno]{amsart}
\usepackage{amssymb,amsrefs}
\usepackage[colorlinks]{hyperref}
\usepackage{amsmath,amsthm,amssymb}
 \usepackage{mathdots}
 \usepackage{graphicx}
 \usepackage{tikz}
  \usepackage{tikz-cd}
 \usetikzlibrary{decorations}
 \usepackage{verbatim}
 \usepackage{booktabs}
  \usepackage{array}
  \usepackage{tabularx}
 \usepackage{booktabs}
\usepackage{siunitx}
\usepackage{rotating}
\usepackage{adjustbox}
\usepackage{longtable}
\usepackage{datetime2}
\usepackage{enumitem}

\newcounter{cases}
\newcounter{subcases}[cases]
\newenvironment{mycase}
{
    \setcounter{cases}{0}
    \setcounter{subcases}{0}
    \newcommand{\case}
    {
        \par\indent\stepcounter{cases}\textbf{Case \thecases.}
    }
    
}
{
    \par
}
\renewcommand*\thecases{\arabic{cases}}

\newcolumntype{C}[1]{>{\centering\arraybackslash}p{#1}}

\newtheorem{thm}{Theorem}[section]
\theoremstyle{theorem}

\theoremstyle{definition}
\newtheorem{example}[thm]{Example}

\theoremstyle{theorem}
\newtheorem{cor}[thm]{Corollary}

\newtheorem{lem}[thm]{Lemma}
\newtheorem{defn}[thm]{Definition}

\theoremstyle{definition}
\newtheorem{remark}[thm]{Remark}

\theoremstyle{theorem}
\newtheorem{prop}[thm]{Proposition}

\DeclareMathOperator{\CM}{CM}
\DeclareMathOperator{\Cl}{Cl}
\DeclareMathOperator{\GL}{GL}
\DeclareMathOperator{\SL}{SL}
\DeclareMathOperator{\PSL}{PSL}
\DeclareMathOperator{\Gal}{Gal}

\DeclareMathOperator{\disc}{disc}
\DeclareMathOperator{\car}{char}
\DeclareMathOperator{\End}{End}
\DeclareMathOperator{\Aut}{Aut}

\DeclareMathOperator{\id}{id}

\DeclareMathOperator{\Split}{Split}

\DeclareMathOperator{\sus}{ss}

\newcommand{\F}{\mathbf{F}}

\newcommand{\Z}{\mathbf{Z}}
\newcommand{\Zee}{\mathbf{Z}}
\newcommand{\Zhat}{\widehat\Zee}
\newcommand{\Q}{\mathbf{Q}}
\newcommand{\Que}{\mathbf{Q}}
\newcommand{\R}{\mathbf{R}}

\newcommand{\OO}{\mathcal{O}}

\newcommand{\tor}{\text{tor}}
\newcommand{\ab}{\text{ab}}

\newcommand{\iso}{\cong}

\newcommand{\pgotic}{\mathfrak{p}}
\newcommand{\gothp}{\mathfrak{p}}

\newcommand{\gothq}{\mathfrak{q}}
\newcommand{\Kbar}{\overline{K}}

\newcommand{\Ogotic}{\mathcal{O}}

\newcommand{\rad}{\text{rad}}

\newcommand{\genlegendre}[4]{%
  \genfrac{(}{)}{}{#1}{#3}{#4}%
  \if\relax\detokenize{#2}\relax\else_{\!#2}\fi
}
\newcommand{\legendre}[3][]{\genlegendre{}{#1}{#2}{#3}}

\newcommand{\tto}{\longrightarrow}
\newcommand{\mapright}[1]
    {{\stackrel{#1}{\longrightarrow}}}

\renewcommand{\mod}{\operatorname{mod}\ }

\usepackage{nicefrac}

\renewcommand{\le}{\leqslant}\renewcommand{\leq}{\leqslant}
\renewcommand{\ge}{\geqslant}\renewcommand{\geq}{\geqslant}
\title[Cyclic reduction densities for Elliptic Curves]
{Cyclic reduction densities for Elliptic Curves}

\subjclass[2010]{11G05; 11R45}

\date{\today}
\keywords{elliptic curves over number fields, cyclic reduction}

\author[F. Campagna, P. Stevenhagen]{Francesco Campagna, Peter Stevenhagen}

\address{ 
Mathematisch Instituut \\ 
Universiteit Leiden   \\ 
Leiden\\
Netherlands}
\email{psh@math.leidenuniv.nl}

\address{
Leibniz Universit\"at Hannover \\
Institut f\"ur Algebra, Zahlentheorie und Diskrete Mathematik\\
Welfengarten 1 \\
30167 
Hannover \\
Germany}
\email{campagna@math.uni-hannover.de}

\begin{document}

\begin{abstract}
For an elliptic curve $E$ defined over a number field $K$, 
the heuristic density of the set of primes of $K$ for which $E$
has cyclic reduction is given by an inclusion-exclusion sum $\delta_{E/K}$ 
involving the degrees of the $m$-division fields $K_m$ of~$E$ over~$K$.
This density can be proved to be correct under assumption of GRH.

For $E$ without complex multiplication (CM), we show that
$\delta_{E/K}$ is the product of an explicit non-negative rational number
reflecting the finite entanglement of the division fields
of $E$ and a universal infinite Artin-type product.
For $E$ admitting CM over $K$ by a quadratic order $\OO$, 
we show that $\delta_{E/K}$ admits a
similar `factorization' in which the Artin type product also depends on $\OO$.

For $E$ admitting CM over $\Kbar$ by an order 
$\OO\not\subset K$, which occurs for $K=\Que$,
the entanglement of division fields over $K$ is non-finite.
In this case we write $\delta_{E/K}$ as the sum
of two contributions coming from the primes of $K$ that are split and inert in $\OO$.
The split contribution can be dealt with by the previous methods, the inert contribution is of a different nature.

We determine the ways in which the density can vanish, and provide numerical examples of the different kinds of densities.

\end{abstract}

\maketitle

\tableofcontents

\section{Introduction}  

\noindent
Let $E$ be an elliptic curve defined over a number field $K$,
and $\gothp$ a prime of $K$ for which~$E$ has good reduction.
Then the point group $E_\gothp(k_\gothp)$
of the reduced curve over the residue class
field $k_\pgotic$ is a finite abelian group on at most two generators.
If one generator suffices, we call $\gothp$ a prime
\emph{of cyclic reduction} of $E$.
The question considered in this paper is whether the set
$S_{E/K}$ of primes of cyclic reduction of $E$ is infinite and,
if so, whether it has a density inside the set of all primes of $K$.

Serre \cite{Serre3} observed in 1977 that this question is very similar
to the Artin primitive root problem (1927),
which asks for the density of the set
of primes $\pgotic$ of $K$ for which a fixed element $a\in K^*$ is a primitive root modulo $\gothp$, i.e., the unit group of the residue class field $k_\gothp$ at $\gothp$ 
satisfies $k_\gothp^*=\langle a \bmod \gothp\rangle$.
In this situation, these primes are (up to finitely many primes of
`bad reduction' for $a$) the primes $\pgotic$ that do not split completely
in any of the `$\ell$-division fields'
$K_\ell=K(\zeta_\ell, \sqrt[\ell]{a})=\Split_K(X^\ell-a)$
of the element $a$ at \emph{prime values} $\ell\in \Z$.
A heuristic inclusion-exclusion argument \cite{Artin}*{pp.viii-ix} that
goes back to Artin suggests that the set of such primes has a
natural density that can be expressed in terms of the degrees of the
$m$-division fields $K_m$ as
\begin{equation}
\label{deltaaK}
\delta_{a, K}=\sum_{m=1}^{\infty} \frac{\mu(m)}{[K_m:K]}.
\end{equation}
Here $\mu$ denotes the M\"obius function.
In 1967, Hooley \cite{Hooley} proved Artin's conjecture for $K=\Que$ 
under the assumption of the Generalized Riemann Hypothesis (GRH).

The set $S_{E/K}$ of primes of cyclic reduction of an elliptic curve
$E/K$ can be characterized in a similar way, in terms of the
\emph{elliptic division fields\/} $K_m=K(E[m](\Kbar))$ of $E$ over $K$
obtained by adjoining to $K$ the coordinates of all $m$-torsion points of $E$
defined over an algebraic closure $\Kbar$ of $K$.
It is, up to finitely many~$\pgotic$, the set of
primes that do not split completely in any of the 
$\ell$-division fields $K_\ell$ of $E$ at prime values $\ell\in \Z$
(Corollary \ref{S_EKdescription}).
Unsurprisingly, the associated heuristic density for $S_{E/K}$
is given by an inclusion-exclusion sum
\begin{equation}
\label{deltaEK}
\delta_{E/K}=\sum_{m=1}^{\infty} \frac{\mu(m)}{[K_m:K]}
\end{equation}
that, at least typographically, is \emph{identical} to \eqref{deltaaK}.
It is the limit of the finite sums
\begin{equation}
\label{delta(n)}
\delta_{E/K}(n)= \sum_{m|n} \frac{\mu(m)}{[K_m:K]}
\end{equation}
for $n$ tending to infinity under the partial ordering of divisibility.
Under this ordering, we have
\begin{equation}
m|n \Longrightarrow \delta_{E/K}(m) \ge \delta_{E/K}(n)\geq 0,
\end{equation}
so the limit exists and is non-negative.
As each value $\delta_{E/K}(n)$ is an upper bound for the
\emph{upper density} of the set $S_{E/K}$, so is $\delta_{E/K}$.
Proving unconditionally that $S_{E/K}$ is infinite in case
$\delta_{E/K}$ in \eqref{deltaEK} is positive is an open problem, just as in
Artin's original setting \eqref{deltaaK}.
If $\delta_{E/K}$ vanishes, then we can show unconditionally that $S_{E/K}$ is finite.

For $K=\Q$, Serre showed that, under GRH, the set $S_{E/\Q}$ does have
a density, and that it is equal to $\delta_{E/\Q}$.
His proof, which is in the spirit of Hooley, was published
in 1983 by Murty~\cite{Murty1}.
It can be extended to general number fields
along lines that are actually already contained 
in the very same paper \cite{Murty1}.
For completeness, we summarize the argument in Section 2.

The number $\delta_{E/K}$ in \eqref{deltaEK}, which we now know to be 
the density (under GRH) of the set
$S_{E/K}$ of primes of $K$ of cyclic reduction,
is defined by a 
series that converges rather slowly, and it is unsuitable to determine whether $\delta_{E/K}$ vanishes.
We remedy this situation by `factoring' the infinite sum as the product
of a finite sum and an infinite non-vanishing product.
The outcome depends on whether $E$ has complex multiplication~(CM)
over an algebraic closure $\Kbar$ of~$K$, leading to the distinction of
three mutually exclusive cases for an elliptic curve $E$ defined over $K$:
\begin{mycase}
            \case $E$ without CM over $\Kbar$;
            \case $E$ with CM that is defined over $K$;
            \case $E$ with CM that is not defined over $K$.
\end{mycase}

\smallskip\noindent 
These cases are the subject of Theorems \ref{factorizationofdeltaEK},
\ref{thm:density containing the CM field} and
\ref{thm:density not containing the CM field}, respectively.
In the first two cases, the finite sum that goes into the `factorization'
of $\delta_{E/K}$ is a sum $\delta_{E/K}(n)$ from~\eqref{delta(n)},
where $n=N(E,K)$ is the product
\begin{equation}
\label{NEK}
N(E,K)=\prod_{\ell\in T_{E/K}} \ell\in\Z_{>0}
\end{equation}
of the primes in an explicit \emph{finite} set $T_{E/K}$ of
\emph{critical prime numbers} for $E$ over $K$.
\medskip

{\bf Case 1.} 
For $E$ without CM over $\Kbar$,
the set $T_{E/K}$ of \emph{critical non-CM primes}
for $E$ consists of those prime numbers $\ell$ 
satisfying at least one of the following conditions:
\begin{enumerate}
\item $\ell \mid 2\cdot 3\cdot 5 \cdot\Delta_K$, 
with $\Delta_K$ the discriminant of $K$;
\item $\ell$ lies below a prime of bad reduction of $E$;
\item the Galois group $\Gal(K_\ell/K)$ is not isomorphic to $\GL_2(\F_\ell)$.
\end{enumerate}
Condition (3) is only satisfied for finitely many $\ell$ by Serre's open image theorem \cite{Serre}.
We prove the following `factorization theorem'.
\begin{thm}
\label{factorizationofdeltaEK}
Let $E/K$ be an elliptic curve without CM, and $n=N(E,K)$ as in~\eqref{NEK}.
Then $\delta_{E/K}$ can be factored as a product
$$
\delta_{E/K}=\delta_{E/K}(n)\cdot
\prod_{\ell \nmid n, \ \ell \ \text{prime}}
\left( 1- \frac{1}{(\ell^2-1)(\ell^2-\ell)} \right)$$
of the finite sum $\delta_{E/K}(n)\in\Que_{\ge0}$ in \eqref{delta(n)}
and a non-vanishing infinite product.
\end{thm}
\noindent
Theorem \ref{factorizationofdeltaEK}
reflects the fact that some \emph{critical $n$-division field} $K_n$ of $E$
and the fields $K_\ell$ for primes $\ell\nmid n$ form
a \emph{linearly disjoint family\/} of finite Galois extensions of $K$,
as in Definition \ref{def:linearlydisjoint} and Theorem \ref{teorema cruciale}.
The existence of such an integer $n$ follows from the open image theorem
in \cite{Serre}, but not the explicit value $n=N(E,K)$ that we provide.
Our $N(E,K)$ will not in general be the smallest such value.

The finite sum $\delta_{E/K}(n)$ in Theorem \ref{factorizationofdeltaEK},
which may vanish, is the density of the set of primes 
of $K$ that do not split completely in any of the 
$\ell$-division fields $K_\ell$ at the critical primes $\ell\in T_{E/K}$.
At primes $\ell\notin T_{E/K}$ the fields $K_\ell$
have \emph{maximal degree}
\begin{equation}
\label{maxdegree}
[K_\ell:K]=\#\GL_2(\F_\ell)= (\ell^2-1)(\ell^2-\ell),
\end{equation}
so the factor at $\ell\notin T_{E/K}$ in the infinite product
is the density of the set of primes not splitting completely in 
$K\subset K_\ell$.
Just as in the case of $\delta_{a,K}$ in \eqref{deltaaK}, we deduce
that $\delta_{E/K}$ can be written as a product
\begin{equation}
\label{delta=cAinfty}
\delta_{E/K}= c_{E/K} \cdot A
\end{equation}
of a rational number $c_{E/K}\in\Q_{\ge0}$ and
a universal \emph{non-CM elliptic Artin constant}
\begin{equation}
\label{Ainfty}
A= \prod_{\ell \ \text{prime}}
        \left( 1- \frac{1}{(\ell^2-1)(\ell^2-\ell)} \right)
        \approx 0.8137519.
\end{equation}
We have $c_{E/K}=1$ when the division fields $K_\ell$ for $\ell$ prime
all assume the maximal degree in~\eqref{maxdegree}
\emph{and} they form a linearly disjoint family over $K$.
When the first condition is not satisfied, a 
decomposition that is more satisfactory than~\eqref{delta=cAinfty} is 
$\delta_{E/K}=\alpha_{E/K}\cdot A_{E/K}$,
where the \emph{naive density}
\begin{equation}
\label{naive density}
A_{E/K} =\prod_{\ell \ \text{prime}} \left( 1- \frac{1}{[K_\ell:K]} \right).
\end{equation}
depends only on the degrees $[K_\ell:K]$,
and the rational \emph{entanglement correction factor}
$\alpha_{E/K}\in\Q_{\ge0}$ reflects the dependencies that exist between
finitely many critical $K_\ell$.

The density $\delta_{E/K}$ vanishes if and only if the finite rational
sum $\delta_{E/K}(n)$ in Theorem~\ref{factorizationofdeltaEK} vanishes.
In this case we know unconditionally that $S_{E/K}$ is a 
finite set, as there is a `finite' obstruction 
to cyclic reduction in the division field $K_n$.
We call this obstruction \emph{non-trivial} if the corresponding naive density
$A_{E/K}$ is non-zero, meaning that for no prime $\ell$,
the $\ell$-division field $K_\ell$ is equal to $K$.
In this case it is the entanglement correction factor $\alpha_{E/K}$ that vanishes.

Non-trivial vanishing does not occur over $K=\Que$,
and `natural' examples are rare. 
We show however that non-trivial vanishing does happens for all elliptic curves $E/K$ with naive density $A_{E/K}>0$, in the sense that over a suitable base change $K\subset K'$, we have $\delta_{E/K'}=0\ne A_{E/K'}$.
We do not know whether non-trivial vanishing can occur 
over $K=\Que(j(E))$.
\medskip

In the {\bf CM-cases 2 and 3}, the endomorphism ring
$\OO=\End_{\Kbar} (E)\ne\Zee$ 
is an imaginary quadratic order.
In this case, the finite set $T_{E/K}$ of \emph{critical CM-primes} of $E/K$
is defined as consisting of those prime numbers $\ell$
for which at least one of the following holds:
\begin{enumerate}
\item
$\ell$ is divisible by a prime of bad reduction of $E/K$, or
\item
$\ell$ divides the product $\Delta(\OO)\cdot\Delta_K$ of 
the discriminants of $\OO$ and $K$.
\end{enumerate}
For CM-curves $E/K$,
the degree $[K_m:K]$ of the division fields $K_m$
only grows quadratically in~$m$, and the ramification
can be controlled by class field theory.
This makes it possible to prove \emph{unconditionally} that~$S_{E/K}$
has density~$\delta_{E/K}$,
as the explicit versions of the Chebotarev density theorem without GRH
in \cite{Lagarias-Montgomery-Odlyzko} now suffice to handle the error terms.
For $K=\Q$, this has been written down in \cite{Cojocaru},
and the proof given there extends without essential changes
to arbitrary number fields~$K$.

\medskip\noindent
{\bf Case 2.}
Suppose that $K$ contains the CM-field $F=\mathrm{Frac}(\OO)$.
In this case, the situation is structurally reminiscent
of the non-CM-case, 
but now almost all $\ell$-division fields have
a Galois group $\Gal(K_\ell/K)\iso (\OO/\ell\OO)^*$ that also depends on $\OO$,
not just on $\ell$ as in the case of the 
generic group $\GL_2(\F_\ell)$ that we had before.
To ease notation, we write $D=\Delta(\OO)$ for the discriminant of $\OO$ and define
the \emph{Artin constant} of the order $\OO$ as
\begin{equation}
\label{AO}
A_\OO=\prod_{\ell \ \text{prime}} A_{\OO,\ell}, 
\end{equation}
with local factors $A_{\OO,\ell}$ at $\ell$ given by 
\begin{equation}
\label{AOl}
A_{\OO,\ell} = 1 - \frac{1}{\# (\OO/\ell\OO)^*} =
\begin{cases}
1-(\ell-1)^{-2} &\mbox{if } \legendre{D}{\ell}=1; \\
1-(\ell^2-1)^{-1} &\mbox{if } \legendre{D}{\ell}=-1; \\
1-(\ell^2-\ell)^{-1} &\mbox{if } \legendre{D}{\ell}=0.
\end{cases}
\end{equation}
The Artin constant for $\OO$ vanishes if and only if we have $A_{\OO,2} =0$, or, equivalently, for$\Delta(\OO)\equiv1\bmod8$.

In the case where $\OO$ (and therefore $F$) is contained in $K$,
the non-maximality and entanglement of the division fields
$K_\ell$ of $E$ over $K$ is `finite', just as in the non-CM case.
We have the following CM-analogue of Theorem \ref{factorizationofdeltaEK}.
\begin{thm}
\label{thm:density containing the CM field}
Suppose $E/K$ has CM by an order $\OO\subset K$, and define
$n=N(E,K)$ as in~\eqref{NEK}.
Then $\delta_{E/K}$ can be factored as the product
\begin{equation} 
\label{eq:density containing the CM field}
\delta_{E/K}=\delta_{E/K}(n)\cdot
\prod_{\ell \nmid n, \ \ell \ \text{prime}}
A_{\OO,\ell}
\end{equation}
of the finite sum $\delta_{E/K}(n)\in\Que_{\ge0}$ from \eqref{delta(n)}
and an infinite product.
\end{thm}
\noindent
In Theorem \ref{thm:density containing the CM field},
we have $\delta_{E/K}=c_{E/K}\cdot A_\OO$ for some $c_{E/K}\in\Que_{\ge0}$.
We deduce that in the case where $K$ contains the CM-order $\OO$,
all densities $\delta_{E/K}$ are rational multiples of the Artin constant
$A_F=A_{\OO_F}$ associated to the ring of integers $\OO_F$ of the CM-field $F$.

\medskip\noindent
{\bf Case 3.}
We finally suppose that $E$ has CM by an order $\OO$
\emph{not} contained in $K$.
Then the entanglement of the division
fields $K_\ell$ is no longer of finite nature, and the
density $\delta_{E/K}$ can sometimes be established
by easy arguments.
\begin{example}
\label{example:delta=1/2}
The elliptic curve $E/\Que$ given by $y^2 = x^3 + x$
has good reduction outside~2, and its CM-order $\OO=\Zee[i]$
of discriminant $D=-4$ is generated by the automorphism
\[
[i]: (x,y)\mapsto (-x,iy)
\]
of $E$ defined over $F=\Que(i)$.
The 2-division field of $E$ is $K_2={\Split}_\Que(x^3+x)=\Que(i)$, and
for $\ell>2$ the $\ell$-division field $K_\ell$ contains $\Que(i)$,
as for any non-zero $\ell$-torsion point $(x,y)\in E(K_\ell)$, we have
$y\ne0$ and $(-x,iy)\in E[\ell](K_\ell)$, hence $i\in K_\ell$.

If $p$ is a prime of cyclic reduction for $E$, then it is odd and does
not split completely in $K_2=\Que(i)$.
Conversely, any odd $p$ that does not split completely in $K_2=\Que(i)$
is a prime of cyclic reduction for $E$, as it does not split completely
in any other $\ell$-division field $K_\ell\supset K_2$.
Thus, the set $\{p: p\equiv 3\bmod 4\}$ of primes of cyclic reduction of $E$
has density $\delta_{E/\Que}=1/2$.
\end{example}

\noindent 
In the case $\OO\not\subset K$ we treat the cyclicity of 
reduction modulo a prime $\gothp\notin T_{E/K}$
according to the splitting behavior of $\gothp$ in the
quadratic extension $K\subset KF$.

For $\gothp$ inert in $K\subset KF$, the reduction 
$E\mod \gothp$ is supersingular, and it follows from
Proposition \ref{prop:containment F in odd-division fields} that in this case,
the cyclicity of the resulting point group $E(k_\gothp)$ \emph{only} depends
on the splitting behavior of $\gothp$ in $K\subset K_2$.
These supersingular $\gothp$ contribute a rational density $\delta_{E/K}^{\sus}$
to $\delta_{E/K}$.
For $\gothp\notin T_{E/K}$ split in $K\subset KF$ as
$\gothp\OO_{KF}=\gothq_1\gothq_2$,
we have $E(k_\gothp)\iso E(k_{\gothq_1})\iso E(k_{\gothq_2})$,
so there is a 2 to 1 correspondence between the set of 
primes $\gothq\notin T_{E/KF}$ of cyclic ordinary reduction for the
base changed elliptic curve $E/KF$
and the set of primes $\gothp\notin T_{E/K}$ of cyclic ordinary reduction
for $E/K$.
\begin{thm}
\label{thm:density not containing the CM field}
Let $E/K$ be a CM-elliptic curve with CM-order $\OO$ and 
CM-field $F\not\subset K$.
Then the cyclic reduction density $\delta_{E/K}$ can be written as
\begin{equation} \label{eq:density not containing the CM field}
\delta_{E/K}=\delta_{E/K}^{\sus} + 
\textstyle \frac{1}{2}\delta_{E/KF},
\end{equation}
where the supersingular density $\delta_{E/K}^{\sus}$ satisfies
$\delta_{E/K}^{\sus}\in\{0, \frac{1}{4}, \frac{1}{2}\}$, and
$\delta_{E/KF}$ admits a factorization as in Theorem \ref{thm:density containing the CM field}.
\end{thm}
\noindent
In cases of Theorem \ref{thm:density not containing the CM field} where
$\delta_{E/KF}$ vanishes but $\delta_{E/K}^{\sus}$ does not, we obtain
examples such as Example \ref{example:delta=1/2}
where $\delta_{E/K}$ is rational and
cyclicity of reduction modulo $\gothp$ only depends on the splitting of
$\gothp$ in $K\subset K_2$.

\bigskip\noindent
{\bf Acknowledgments. }
The first author is supported by ANR-20-CE40-0003 Jinvariant.
The second author was funded by a research grant of the 
Max-Planck-Institut f\"ur Mathematik in Bonn.
Both authors thank the institute in Bonn for 
its financial support and its inspiring atmosphere.

\section{The density of the set of primes of cyclic reduction}

\noindent
This section gives the basic splitting criterion (Corollary \ref{S_EKdescription})
for primes of $K$ to be primes of cyclic reduction for $E/K$,
and also a proof that for arbitrary number fields $K$,
the density of 
cyclic reduction is given as in \eqref{deltaEK} by the 
inclusion-exclusion sum
\begin{equation}
\label{deltasum}
\delta_{E/K}=\sum_{m=1}^{\infty} \frac{\mu(m)}{[K_m:K]}.
\end{equation}
We define our elliptic curve $E/K$ by an integral Weierstrass equation
$y^2=x^3+Ax+B$ with coefficients $A, B$ in the ring of integers $\OO_K$
of the number field $K$.
To this model of $E$ we associate the discriminant
\begin{equation}
\label{Delta_E}
\Delta_E= -16(4A^3+27B^2) \in\OO_K\setminus\{0\}.
\end{equation}
The primes $\gothp$ of $K$ coprime to $\Delta_E$, which are
necessarily of odd characteristic,
are the primes \emph{of good reduction} for this model of~$E$.
Changing the model, or providing a minimal model,
will only affect the reduction type of finitely many primes,
and this is irrelevant for our density questions.
For primes $\gothp$ of good reduction for $E/K$,
the reduced curve modulo~$\gothp$ is an elliptic curve
$E_\gothp$ over the residue class field~$k_\gothp$.
We begin by formally stating the key criterion for a prime of good
reduction of $E/K$ to be a \emph{prime of cyclic reduction} of $E/K$, i.e.,
a prime for which the finite group $E_\gothp(k_\gothp)$ is cyclic.

\begin{lem}
\label{cyclicreductioncriterion}
For a prime $\gothp$ of good reduction for the elliptic curve $E/K$,
the following are equivalent:
\begin{enumerate}
\item $\gothp$ is a prime of cyclic reduction for $E/K$;
\item for no prime number $\ell$ coprime to $\gothp$,
the prime $\gothp$ splits completely in $K\subset K_\ell$,
with $K_\ell$ the $\ell$-division field of~$E$.
\end{enumerate}
\end{lem}

\begin{proof}
Pick $\gothp\nmid\Delta_E$.
Then $E_\gothp(k_\gothp)$ is a cyclic group if and only if for no
prime $\ell$, its $\ell$-torsion subgroup $E_\gothp[\ell](k_\gothp)$
has order $\ell^2$.
For $\ell=\car(k_\gothp)$, it is a generality on elliptic curves
in positive characteristic that the group $E_\gothp[\ell](k_\gothp)$
is cyclic, so we can further assume $\gothp\nmid\ell$.
Then $\gothp$ is unramified in the Galois extension $K\subset K_\ell$,
as it is a prime of good reduction of $E$ coprime to $\ell$.

The group $E[\ell](K_\ell)$ has order $\ell^2$ by definition of $K_\ell$,
and at every prime $\gothq|\gothp$ of $K_\ell$, the natural reduction map
$E[\ell](K_\ell)\to E_\gothq(k_\gothq)$ is injective as
$\gothq\nmid \ell\Delta_E$ is a prime of good reduction of $E$ in $K_\ell$.
Thus $E_\gothq[\ell](k_\gothq)$ has order $\ell^2$.
Now $k_\gothq$ is generated over $k_\gothp$ by the coordinates of the
points in $E_\gothq[\ell](k_\gothq)$, as $K\subset K_\ell$ is generated by the
coordinates of the $\ell$-torsion points of $E$.
It follows that the natural inclusion $k_\gothp\subset k_\gothq$ is an
equality for all $\gothq|\gothp$ in $K_\ell$, i.e.,
$\gothp$ splits completely in $K\subset K_\ell$,
if and only if the natural inclusion
$E_\gothp[\ell](k_\gothp)\subset E_\gothq[\ell](k_\gothq)$
is an equality.
As $E_\gothq[\ell](k_\gothq)$ has order $\ell^2$, this proves the lemma.
\end{proof}

\noindent
If $\gothp$ is a prime of good reduction of $E$ of characteristic
$p$ coprime to the discriminant $\Delta_K$ of $K$,
then $\gothp$ can not split completely
in the division field $K_p$, as it is totally ramified in the subextension
$K\subset K(\zeta_p)$ of degree $p-1>1$ of $K \subset K_p$ that is generated
by a primitive $p$-th root of unity $\zeta_p$.
This shows that, for primes $\gothp$ coprime to both $\Delta_E$
and $\Delta_K$, being in the set $S_{E/K}$ of primes of cyclic
reduction of $E$ is tantamount to \emph{not} splitting completely in \emph{any} division field extension $K\subset K_\ell$ at a rational prime $\ell$.

\begin{cor}
\label{S_EKdescription}
For a prime $\gothp\nmid \Delta_E\Delta_K$, we have $\gothp\in S_{E/K}$
if and only if $\gothp$ does not split completely in any of the
division fields $K_\ell$, with $\ell\in\Z$ prime.
\qed
\end{cor}

\noindent
The proof of Lemma \ref{cyclicreductioncriterion} shows that if a prime
$\gothp\nmid \Delta_E\Delta_K$ splits completely in $K\subset K_\ell$, then
$E_\gothp(k_\gothp)$ has complete $\ell$-torsion, so we have
$ \ell\le \sqrt {N_{K/\Q}(\gothp)}+1 $
by the Hasse-Weil bound.
For a squarefree integer $m$ and a prime $\gothp \nmid \Delta_E$,
we similarly obtain
\begin{equation}
\label{HWbound}
\text{$\gothp\nmid \Delta_E\Delta_K$ splits completely in $K\subset K_m$}
\ \Longrightarrow \ 
m\le \sqrt {N_{K/\Q}(\gothp)}+1 .
\end{equation}

\noindent
In order to count the cardinality $\# S_{E/K}(x)$ of primes $\gothp$ in
the set $S_{E/K}$ of cyclic reduction of norm
$N_{K/\Q}(\gothp)\le x\in\R_{>0}$,
we introduce the counting function
$$
\pi_K(x, K_m)=\#\{\gothp\nmid\Delta_E\Delta_K:\
    N_{K/\Q}(\gothp)\le x \text{\ and\ }
    \gothp \text{\ splits completely in\ }K\subset K_m\}.
$$
The function $\pi_K(x, K)$ counts primes $\gothp\nmid\Delta_K$
of good reduction of $E$ of norm at most~$x$, and,
as there are only finitely many primes $\gothp|\Delta_K$ in $S_{E/K}$,
Corollary \ref{S_EKdescription} and inclusion-exclusion yield
\begin{equation}
\label{cyclicprimes<x}
\# S_{E/K}(x)=
\sum_{m=1}^\infty \mu(m) \pi_K(x, K_m) + {\mathrm O}(1).
\end{equation}
Note that by \eqref{HWbound}, the function $\pi_K(x, K_m)$ vanishes for
$m>\sqrt x+1$, so the infinite sum of integers in \eqref{cyclicprimes<x}
is actually finite, and therefore convergent.

In order to obtain the desired asymptotic relation
$\# S_{E/K}(x)\sim \delta_{E/K} \cdot x/\log x$ 
we use the asymptotic relations
$\pi_K(x, K_m)\sim \frac{1}{[K_m:K]} \cdot x/\log x$.
After dividing both sides in \eqref{cyclicprimes<x} by $x/\log x$,
the problem in obtaining \eqref{deltasum} lies in interchanging the infinite sum and the limit
$x\to\infty$ in the right hand side of \eqref{cyclicprimes<x}.
This can be done if one assumes GRH in order to bound the error terms 
in the asymptotic relations for $\pi_K(x, K_m)$ and adapts the argument
of Hooley found in~\cite{Hooley}.
More precisely, Murty \cite{Murty1}*{Theorem 1} has shown that in this setting,
for the density of $S_{E/K}$ to be equal (under GRH) to the 
the inclusion-exclusion density \eqref{deltaEK} it is enough to show that
$[K_m:K]$ grows sufficiently rapidly with $m$ (as it does for division fields)
and that \emph{two} conditions are satisfied.
The first condition is that the root discriminants of the division fields
$K_m$ do not grow too rapidly with $m$, as follows.

\begin{prop}
\label{rootdiscriminantKm}
For $m\in \Z_{>0}$ tending to infinity, we have
$$\frac{1}{[K_m:K]}\log|\Delta_{K_m}|=\mathrm{O}(\log m)$$
\end{prop}

\noindent
Note that the quantity in the Proposition is $[K:\Q]$ times the
logarithm of the ordinary root discriminant $|\Delta_{K_m}|^{1/[K_m:\Q]}$.

The second condition is that `not too many' primes $\gothp$ of $K$
split in the division fields~$K_\ell$ for `large' primes $\ell$,
in the following sense.
\begin{prop}
\label{largeell}
The number of primes $\gothp$ of $K$ of
norm $N_{K/\Q}(\gothp)\le x$ that split completely in
$K\subset K_\ell$ for some prime $\ell>\frac{x^{1/2}}{\log^2 x}$
is $\mathrm{o}(\frac{x}{\log x})$ for $x\to \infty$.
\end{prop}
\noindent
Thus, in order to establish that, under GRH, the quantity
$\delta_{E/K}$ in \eqref{deltaEK} is in all cases
the density of the set $S_{E/K}$ of primes of cyclic reduction,
it now suffices to prove Propositions \ref{rootdiscriminantKm}
and \ref{largeell}.
For completeness, we sketch the proofs in the remainder of this Section.
\begin{proof}[Proof of Proposition \ref{rootdiscriminantKm}]
Bounding absolute root discriminants already dates back to Hensel
\cite{Serre2}*{p. 58}.
For the relative extension $K\subset K_m$ we can use
the version found in \cite{Murty2}*{p.\ 44}.
It states that for a finite Galois extension of number fields
$K\subset L$ with relative discriminant $\Delta_{L/K}$
of norm $D(L/K)=N_{K/\Q}(\Delta_{L/K})\in\Z_{>0}$,
we have
\begin{equation}
\label{relativeHensel}
\log D(L/K) \leq ([L:\Q]-[K:\Q]) \sum_{p| D(L/K)} \log p + [L:\Q] \log([L:K]).
\end{equation}
As the absolute discriminant of $L$ equals
$|\Delta_L|=D(L/K)|\Delta_K|^{[L:K]}$,
the identity
$$
\log D(L/K)=\log|\Delta_L|-[L:K]\log|\Delta_K|
$$
can be combined with the inequality \eqref{relativeHensel} in the case
$L=K_m$ to obtain, after division by $d_m=[K_m:K]$, the estimate

\begin{align*}
\frac{1}{d_m}\log |\Delta_{K_m}| -  \log |\Delta_K|
        &\leq  \frac{[K_m: \Q]-[K:\Q]}{d_m}
                \sum_{p|D(K_m/K)}~\log p  + \frac{ [K_m:\Q] }{d_m}\log d_m \\
        & \leq [K:\Q] \cdot \bigl(\sum_{p|D(K_m/K)} \log p + \log d_m \bigr).
\end{align*}
The primes $p|D(K_m/K)$ either divide $m$, or they lie under one of
the finitely many primes of bad reduction of $E$, so we have
$\sum_{p|D(K_m/K)} \log p\le C_E + \log m$ for some constant $C_E$
depending only on $E$.
We obtain
$$
\frac{1}{d_m} \log|\Delta_{K_m}| \le
        [K:\Q] \cdot \bigl(\log |\Delta_K| + C_E + \log m + \log d_m\bigr).
$$
As we have $d_m=\text{O}(m^4)$, this yields the desired asymptotic relation.
\end{proof}

\begin{proof}[Proof of Proposition \ref{largeell}]
When showing that the cardinality of the set of primes $\gothp$ of~$K$ of
norm $N_{K/\Q}(\gothp)\le x$ that split completely in $K\subset K_\ell$
for some prime $\ell>\frac{x^{1/2}}{\log^2 x}$ is asymptotically
$\mathrm{o}(\frac{x}{\log x})$, we may disregard primes
$\gothp|\Delta_K\Delta_E$, as they are finite in number,
and primes $\gothp$ that are not of degree 1,
as there are no more than $\mathrm{o}(\sqrt x)$ of them.

Suppose now that $\gothp\nmid\Delta_K\Delta_E$ is of
prime norm $N_{K/\Q}(\gothp)=p\le x$, and that $\gothp$
splits completely in an $\ell$-division field $K_\ell$ with $\ell>2$.
By \eqref{HWbound}, this implies $\ell\le \sqrt x +1$.
As $\gothp\nmid\Delta_K$ necessarily splits completely in
the subextension $K\subset K(\zeta_\ell)$, we have $\gothp\nmid\ell$,
and $\gothp$ lies over a rational prime $p\equiv 1\bmod \ell$.
Any such $p$ gives rise to at most $[K:\Q]$ primes $\gothp$ in $K$ of
norm $p$.
Thus, the number $B(x)$ of such $\gothp$ can be bounded by
\begin{equation}
\label{p1modell}
B(x)\le
[K:\Q]
\cdot \sum_{\frac{x^{1/2}}{\log^2 x} < \ell < x^{1/2}+1}
\pi(x, 1, \ell),
\end{equation}
with $\pi(x, 1, \ell)$ denoting the number of primes $p\le x$
satisfying $p\equiv 1\bmod \ell$.
By the Brun-Titchmarsh theorem, we have
$$
\pi(x,1,\ell) \leq
        \frac{2x}{\varphi(\ell) \log (\frac{x}{\ell})} \ll
        \frac{x}{\ell \log (\frac{x}{\ell})},
$$
so we obtain
$$
B(x)\ll \sum_{\frac{x^{1/2}}{\log^2 x} < \ell < x^{1/2}+1}
\frac{x}{\ell \log (\frac{x}{\ell})}\ll
\frac{x}{\log(x)} \cdot \sum_{\frac{x^{1/2}}{\log^2 x} < \ell < x^{1/2}+1}
                \frac{1}{\ell}.
$$
Applying the well known estimate \cite{Apostol}*{Theorem 4.12} for the sum
$\sum_{\ell<X} \ell^{-1}$ of reciprocals of primes, we obtain
$B(x)=\mathrm{o}(\frac{x}{\log x})$ for $x\to \infty$. 
%
\end{proof}

\section{Cyclic reduction density in the non-CM case}

\noindent
The explicit definition \eqref{deltaEK} of
the conjectural density $\delta_{E/K}$ of cyclic reduction
is ill-suited to determine either its approximate value or its vanishing.
Theorem \ref{factorizationofdeltaEK} allows us to approximate
it with high precision from fewer data, and to determine
whether it vanishes.
Its proof amounts to determining the entanglement of the
division fields $K_\ell$, which, by the open image theorem of Serre
\cite{Serre}, is `of finite nature' for $E$ without CM.
\begin{defn}
\label{def:linearlydisjoint}
Let $F$ be a field and let $\mathcal{F}=\{F_n\}_{n\in X}$
a family of Galois extensions of $F$ inside an algebraic closure of $F$.
We call $\mathcal{F}$ linearly disjoint over $F$ if for the
compositum $L$ of the fields $F_n$,
the natural inclusion map
$$\varphi: \Gal(L/F) \hookrightarrow \prod_{n\in X} \Gal(F_n/F)$$
is an isomorphism.
If this is not the case, we call the family $\mathcal{F}$
entangled over $K$.
The entanglement is said to be finite if $\varphi[\Gal(L/F)]$ is of finite index in $\prod_{n\in X} \Gal(F_n/F)$.
\end{defn}

\noindent
A family as in Definition \ref{def:linearlydisjoint}, which can be either finite
or infinite, depending on $X$,
is linearly disjoint over $F$ if and only if for every individual
field $F_n\in \mathcal{F}$,
the field $F_n$ and the compositum of the fields $F_m$ with $m\ne n$
are linearly disjoint over $F$.

The family of $\ell$-division fields $K_\ell$ with $\ell$ prime is
not necessarily linearly disjoint over a number field $K$
for an elliptic curve $E/K$ without CM,
but we can make it linearly disjoint by grouping together
finitely many `critical' $K_\ell$ in their compositum.
This crucial fact was already formulated in the Introduction,
and we repeat it here.
\begin{thm} \label{teorema cruciale}
Let $K$ be a number field, $E/K$ an elliptic curve without CM,
and define the set $T_{E/K}$ of critical primes for $E/K$ as
consisting of the prime numbers $\ell$ satisfying at least one of
the following conditions:
\begin{enumerate}
\item $\ell \mid 2\cdot 3\cdot 5 \cdot\Delta_K$;
\item $\ell$ lies below a prime of bad reduction of $E$;
\item the Galois group $\Gal(K_\ell/K)$ is not isomorphic to $\GL_2(\F_\ell)$.
\end{enumerate}
If $N\in\Zee_{>0}$ is divisible by all primes in $T_{E/K}$,
then the family consisting of $K_N$ and $\{K_\ell \}_{\ell \nmid N}$
is linearly disjoint over $K$.
\end{thm}

\noindent
The proof of Theorem \ref{teorema cruciale} relies on a
group theoretical result on the Jordan-H\"older factors
that can occur in subgroups of $H\subset \GL_2(\Z/N\Z)$.

\begin{lem}\label{group lemma}
Let $N\in\Z_{>0}$ be an integer and $H\subset \GL_2(\Z/N\Z)$ a subgroup.
Suppose $S$ is a non-abelian simple group that occurs in $H$.
Then $S$ is isomorphic to either $A_5$ or $\PSL_2(\F_\ell)$,
with $\ell$ a prime dividing $N$.
\end{lem}

\begin{proof}
We may assume $H\subset \SL_2(\Z/N\Z)=\prod_{\ell|N} \SL_2(\Z/\ell\Z)$
as we only care about non-abelian simple Jordan-H\"older factors.
In addition, we may assume that $N$ is squarefree, i.e.,
equal to its own radical $N_0=\rad(N)$; indeed, the natural
map $r:\SL_2(\Z/N\Z) \to \SL_2(\Z/N_0\Z)$ has a solvable kernel that is a
product of $\ell$-groups, and the groups
$H$ and $H/(H\cap\ker r)\subset \SL_2(\Z/N_0\Z)$
$H$ and $H/(H\cap\ker r)\subset \SL_2(\Z/N_0\Z)$
have the same non-abelian simple Jordan-H\"older factors.
Thus, every non-abelian simple group that occurs in $H$ occurs
in a subgroup of some $\SL_2(\Z/\ell\Z)$, so we can reduce to the case
that $N=\ell$ is prime.
In this case the statement is a classical result that can be
found in \cite{Serre4}*{p. IV-23}.
\end{proof}

\begin{proof}[Proof of Theorem \ref{teorema cruciale}]
It suffices to show that for an integer $N>1$ divisible by all
primes in $T_{E/K}$ and $l\nmid N$ a prime number,
we have $K_N\cap K_\ell=K$.
Take such $N$ and~$\ell$.
Then $\ell\nmid N$ is unramified in the tower $\Q\subset K\subset K_N$
by (1) and (2), and since $K\subset K(\zeta_\ell)$ is totally ramified
over $\ell$ of degree $\ell-1>1$ by (1),
the fields  $K_N$ and $K(\zeta_\ell)$ are $K$-linearly disjoint.
It now suffices to prove that the normal extensions
$K(\zeta_\ell)\subset K_\ell$ and $K(\zeta_\ell)\subset K_N(\zeta_\ell)$
are linearly disjoint over $K(\zeta_\ell)$.

We have $\Gal(K_\ell/K(\zeta_\ell))\cong \SL_2(\F_\ell)$ by (3),
and for $\ell\ge5$ this group has a unique non-trivial
normal subgroup $\{\pm \id_\ell\}$
with simple quotient $\PSL_2(\F_\ell)$.
If $K_\ell \cap K_{N}(\zeta_\ell)$ is not equal to $K(\zeta_\ell)$, we find that
the non-abelian simple group $\PSL_2(\F_\ell)$, which is not $A_5$
as we assume $\ell\ne5$ by (1), is a Jordan-H\"older factor
of $\Gal(K_N(\zeta_\ell)/K(\zeta_\ell))\iso \Gal(K_N/K)$.
As we can view $\Gal(K_N/K)$ as a subgroup of $\GL_2(\Z/N\Z)$, this contradicts
Lemma \ref{group lemma}, since we have $\ell\nmid N$.
\end{proof}

\noindent
We are now ready to prove our main result for 
{\bf Case 1}.

\begin{proof}[Proof of Theorem \ref{factorizationofdeltaEK}]
We simply note that
the quantity $\delta_{E/K}(n)$ in \eqref{delta(n)}
is the inclusion-exclusion fraction of elements in the
Galois group $\Gal(K_n/K)$ that
have non-trivial restriction on every subfield $K_\ell$ with $\ell|N$.
Thus, if $n_1$ and $n_2$ are coprime numbers for which
the division fields $K_{n_1}$ and $K_{n_2}$ are $K$-linearly disjoint, we have
an equality $\delta_{E/K}(n_1n_2)=\delta_{E/K}(n_1)\delta_{E/K}(n_2)$.
If $N'$ is any squarefree multiple of the integer $N$ in
Theorem \ref{teorema cruciale}, this yields
$$
\delta_{N'}(E)= \delta_N(E) \cdot
        \prod_{\ell \mid N'/N, \ \ell \ \text{prime}}
        \left( 1- \frac{1}{[K_\ell:K]} \right).
$$
Taking the limit $N'\to\infty$ with respect to the divisibility ordering
yields Theorem \ref{factorizationofdeltaEK}.
\end{proof}

\noindent
For our purposes, we only need to apply Theorem \ref{teorema cruciale}
for squarefree values of $N$.
We can however strengthen its conclusion a bit and reformulate it
in the following way, as an explicit form of Serre's open image theorem.
This is the form used by Lombardo and Tronto in \cite{Lombardo}.

\begin{thm}
Let $K$ be a number field, $E/K$ an elliptic curve without CM,
and $T_{E/K}$ the set of prime numbers $\ell$ satisfying one of
\begin{enumerate}
\item $\ell \mid 2\cdot 3\cdot 5 \cdot\Delta_K$;
\item $\ell$ lies below a prime of bad reduction of $E$;
\item the Galois group $\Gal(K_\ell/K)$ is not isomorphic to $\GL_2(\F_\ell)$.
\end{enumerate}
For $\ell$ prime, write $K_{\ell^\infty}$ for the compositum of all
$\ell$-power division fields of $E$ over $K$, and
$K_T$ for the compositum of the fields $K_{\ell^\infty}$ with $\ell\in T_{E/K}$.
Then the family consisting of $K_T$ and $\{K_{\ell^\infty}\}_{\ell \notin T_{E/K}}$
is linearly disjoint over $K$.
\end{thm}
\begin{proof}
It suffices to show that for $N$ an integer divisible by all
primes in $T_{E/K}$ and $\ell\nmid N$ prime,
we have $K_N\cap K_{\ell^n}=K$ for every $n\in\Z_{>0}$.
For $n=1$, this is Theorem \ref{teorema cruciale}.

As $K\subset K_N$ is unramified over $\ell\nmid N$ by condition (2),
the intersection is $K$-linearly disjoint from
$K(\zeta_{\ell^n})$ by the condition $\ell\nmid\Delta_K$ in (1),
and it is $K$-linearly disjoint from $K_\ell$ by Theorem \ref{teorema cruciale}.
It therefore corresponds to a subgroup of
$\Gal(K_{\ell^n}/K)\subset \GL_2(\Z/\ell^n\Z)$ that maps
surjectively to $\Gal(K_\ell/K)=\GL_2(\F_\ell)$ by (3) and has
surjective image $(\Z/\ell^n\Z)^*$ under the determinant map.
By a result of Serre \cite{Serre4}*{p. IV-23, Lemma 3}, valid for $\ell\ge5$,
such a group is the full group $\GL_2(\Z/\ell^n\Z)$,
proving $K_N\cap K_{\ell^n}=K$.
\end{proof}

\section{Cyclic reduction densities for CM-curves}  \label{sec:CM_cycred}

\noindent
We now let $E$ be an elliptic curve defined over a number field $K$
for which $\OO=\End_{\Kbar} (E)$ is an imaginary quadratic order of
discriminant $\Delta(\OO)<0$.
In other words, the $j$-invariant $j_E\in K$ of $E$ and the
analytically defined $j$-invariant $j(\OO)$ of the order $\OO$
are algebraic numbers having the same irreducible 
polynomial over $\Que$.
We write $\Delta(\OO)=f^2 \Delta_F$, where $\Delta_F$ denotes the discriminant of the CM-field 
$F=\mathrm{Frac}(\OO)=\Que(\sqrt {\Delta(\OO)})$ of $E$
and the \emph{conductor} $f\in\Zee_{>0}$ the index of $\OO$ in the ring of integers $\OO_F$ of $F$.

If $F$ is contained in $K$,
the basic assumption of {bf Case 2},
the situation with respect to cyclic reduction resembles the non-CM case:
entanglement of the $\ell$-division fields $K_\ell$ concerns
only finitely many $\ell$, and for almost all $\ell$, the Galois group
of the extension $K\subset K_\ell$ is the full group $(\OO/\ell\OO)^*$.
The proof of the main result, Theorem \ref{thm:density containing the CM field}, will therefore be short.

\begin{proof}[Proof of Theorem \ref{thm:density containing the CM field}]
Just as in the proof of Theorem \ref{factorizationofdeltaEK}, it 
suffices to have the analogue of Theorem \ref{teorema cruciale} in the 
CM-case with $\OO\subset K$.
As in the proof of this Theorem, it suffices to show that for $N$
any integer divisible by the critical CM-primes that make up $T_{E/K}$ and
$\ell\nmid N$ prime, we have $K_N\cap K_\ell= K$.
This can be shown by a ramification argument.
On the one hand,
the extension $K\subset K_N$ is unramified at all primes over $\ell$,
as these are primes of good reduction of $E$ and $\ell$ does not divide $N$.
On the other hand, Campagna and Pengo show in 
\cite{Campagna-Pengo}*{Proposition 3.3} that any subextension $K\subset L$ of
$K\subset K_\ell$ is ramified above $\ell$ for $L\ne K$.
Moreover, we have $\Gal(K_\ell/K)=(\OO/\ell\OO)^*$ 
at these $\ell$, so the factor corresponding to $\ell$ in the infinite product  
\eqref{eq:density containing the CM field}
is $A_{\OO,\ell}$ as stated.
\end{proof}
\noindent
The argument of Campagna and Pengo mentioned above uses the fact that
$\ell$ does not divide the discriminant $\Delta_K$ of $K$ to obtain the
desired ramification behavior of $K\subset K_\ell$ over $\ell$.
As $\ell\nmid \Delta_K$ is not a necessary condition for this,
such $\ell$ can in many examples be omitted from $T_{E/K}$.
Note also that in the case $\OO\subset K$, all odd primes $\ell$
dividing $\Delta(\OO)$ will actually divide $\Delta_K$,
as not only $F=\mathrm{Frac}(\OO)$ is contained in $K$, but
also the ring class field $F(j(\OO))$ of the order $\OO$,
which is ramified over $F$ at all odd primes dividing the conductor $f$.

We finally come to {\bf Case 3}, when $\OO$ and $F$ are
\emph{not} contained in $K$.
This situation arises in particular for $K=\Que$. 
In this case the entanglement between the division fields $K\subset K_\ell$
is of a different, non-finite nature, as illustrated by the following example.

\begin{example} \label{example_cm}
    Consider the elliptic curve $E:y^2=x^3+x^2-3x+1$ over $K=\mathbf{Q}$, with $j$-invariant 
    $j_E=8000$ and complex multiplication by $\OO=\mathbb{Z}[\sqrt{-2}]$. 
    By Proposition \ref{prop:containment F in odd-division fields}(1) below, the CM-field $F=\Que(\sqrt {-2})$ is contained in the division field $K_\ell$ for every prime $\ell \geq 3$.
    It is not contained in the 2-division field $K_2=\Que(\sqrt 2)$ generated by the roots of 
    $x^3+x^2-3x+1=(x-1)((x+1)^2-2)$.
    We deduce that $E(\F_p)$ is non-cyclic with complete 2-torsion for primes $p\equiv \pm 1 \bmod 8$,
    whereas primes $p\equiv 5 \bmod 8$ are of cyclic reduction as they 
    do not split completely in any $\ell$-division field $K_\ell$.

    The primes $p\equiv 5 \bmod 8$ cease to be primes of cyclic reduction for $E$ 
    over $F=\Que(\sqrt{-2})$, as they are inert in $F$ and $E$ acquires 
    complete 2-torsion over $\F_{p^2}$.
    Primes $p\equiv 3\bmod 8$ are split in $F$ and give rise to \emph{two} primes $\gothp | p$ in $F$ with $E(k_\gothp)=E(\F_p)$.  
    This yields
    \begin{equation}
    \label{prevdensity}
    \delta_{E/\mathbf{Q}}=\textstyle \frac{1}{4} + \frac{1}{2}\delta_{E/F}.
    \end{equation}
    We now base-change $E$ to $L=\mathbf{Q}(i)$. 
    In this case $E$ has complete 2-torsion modulo primes $p\equiv 3\bmod 4$.
    The primes $p\equiv 1\bmod 4$ split into 2 primes, which are 
    of cyclic reduction for $p\equiv 5 \bmod 8$, 
    and have complete 2-torsion for $p\equiv 1\bmod 8$.
    We conclude that a prime $\gothp$ of $\Que(i)$ is a prime of cyclic reduction for $E$ if and only if its norm is a prime $p\equiv 1\bmod 8$, and find
    \[
    \delta_{E/L}=\textstyle\frac{1}{2}.
    \]
    Note that the cyclic reduction density of $E$ is bigger over $F$ than over $\Que$:
    the $\frac{1}{4}$ in \eqref{prevdensity} from rational primes $p\equiv5\bmod8$ becomes $\frac{1}{2}$, 
    and the $\frac{1}{2}\delta_{E/F}$ from rational primes 
    $p\equiv3\bmod8$ disappears as these primes are 
    inert in $L$: over their residue fields of order $p^2$ in $L$, the reduced elliptic curve has complete 2-torsion, since $p$ splits completely in $L\subset L_2=L(\sqrt{-2})$.
\end{example}

\noindent
Example \ref{example_cm} illustrates the much more general phenomenon concerning the relation between the CM-field $F$ and the $\ell$-division fields $K_\ell$.
\begin{prop}
\label{prop:containment F in odd-division fields}
Suppose that $E/K$ has $\CM$-order $\OO$ and $\CM$-field $F$. Then
\begin{itemize}
\item[\rm (1)]
the $\ell$-division field $K_\ell$ contains $F$ for all primes $\ell\ge 3$;
\item[\rm (2)]
the $2$-division field $K_2$ equals $KF$ for $\Delta(\OO)\equiv 1\bmod 8$;
\item[\rm (3)]
the $2$-division field $K_2$ contains $F$ for $\Delta(\OO)\equiv 5\bmod 8$;
\item[\rm (4)]
for $F\not\subset K$ and $\Delta(\OO)\equiv 0\bmod 4$,
the $2$-division field $K_2$ only contains $F$ in the special case
$\Delta(\OO)=-4$, with $E$ admitting a Weierstrass model 
$Y^2=X^3+aX$ with $a\in {K^*}^2$.
\end{itemize}
\end{prop}

\begin{proof}
Suppose we have $F \not \subset K_\ell$ for a prime number $\ell$.
Then the quadratic extension $K\subset KF$ and the extension
$K\subset K_\ell$ are $K$-linearly disjoint Galois extensions of $K$.
By the Chebotarev Density Theorem,
we can pick a degree $1$ prime $\gothp |p$ of $K$ of good reduction for $E$
that is coprime to $\ell$, inert in $K\subset KF$
and totally split in $K\subset K_\ell$.

The splitting condition in $K\subset KF$ implies,
by \cite{Lang}*{Chapter 13, Theorem 12},
that the reduced curve $E_\gothp$ is a
supersingular elliptic curve over $k_\gothp=\mathbf{F}_p$,
so it has $\# E_\gothp(k_\gothp)=p+1$ points.
The splitting of $\gothp$ in $K\subset K_\ell$ implies that $k_\gothp^*$ 
contains a primitive $\ell$-th root of unity and that 
$E_\gothp(k_\gothp)$ has full $\ell$-torsion.
The resulting divisibility relations
\begin{equation}
\label{prop: divisibilities}
\ell \mid p-1 \qquad\hbox{and}\qquad \ell^2 \mid p+1
\end{equation}
show that we have $\ell=2$, proving (1).

For discriminants $\Delta(\OO)\equiv 1\bmod8$ we have $\Delta(\OO)<-4$ or,
equivalently, $\OO^*=\{\pm1\}$.
In this generic case, which only excludes the special discriminant values 
$-3$ and $-4$, the 2-division field $K_2$ is invariant under twisting of $E$.
Replacing $E$ by a twist, if necessary, we will assume that 
$E$ is defined over $L=\Que(j_E)$ and with CM-order $\OO$.
Note that $L$ does not contain $F$, as $L$ has a real embedding
sending $j_E$ to $j(\OO)$.

From the theory of complex multiplication, 
we obtain the diagram of fields below
in the generic case $\OO^*=\{\pm1\}$.
In this diagram
$LF=F(j_E)=F(j(\OO))$ is the ring class field $H_\OO$ of $\OO$,
with Galois group $\Cl(\OO)$ over $F$.
It is unramified over $F$ outside the conductor of $\OO$, and
unramified over $\Que$ outside $\Delta(\OO)$.
As we have $\OO^*=\{\pm1\}$, the ray class field $H_{2,\OO}$
of conductor 2 of the order $\OO$ has Galois group $(\OO/2\OO)^*$ over $H_\OO$.
It can be obtained as the compositum $L_2F$ of $F$
with the 2-division field $L_2$ of $E$ over $L$.

\[\begin{tikzcd}
                                          & H_{2,\Ogotic} \arrow[rd, no head, "(\OO/2\OO)^*"]                     &                       &   \\
L_2 \arrow[rd, no head] \arrow[ru, no head] &                                           & H_{\Ogotic} \arrow[rd, no head, "\Cl(\OO)"] &   \\
                                          & L \arrow[rd, no head] \arrow[ru, no head,"2"] &                       & F\\
                                          &                                           & \Q \arrow[ru, no head,"2"] &
\end{tikzcd}\]

In the case $\Delta(\OO)\equiv 1\bmod8$, the ring $\OO/2\OO\iso \F_2\times \F_2$
has trivial unit group, so we have $H_{2,\OO}=H_\OO$ and two inclusions
$L\subset L_2\subset LF=H_\OO$ of which exactly one is an equality.
If we have $F\not\subset L_2$, the resulting equality $L=L_2$ implies 
that, trivially, all primes split completely in $L\subset L_2$.
In this case any odd prime $\gothp$ of $L$ of degree 1 that is inert in the
quadratic extension $L\subset LF$ and a prime of good reduction for $E$
satisfies $p=N_{L/\Que}\gothp\equiv -1\bmod4$ by~\eqref{prop: divisibilities}.
This implies $LF=L(i)$, so $H_\OO=LF$ contains $i$.
As the extension $\Que\subset H_\OO$ is unramified outside 
$\Delta(\OO)\equiv 1\bmod8$, this is a contradiction,
showing $F\subset L_2=LF$, hence $F\subset K_2=KF$, proving (2).

For $\Delta(\OO)\equiv 5\bmod8$, the ring $\OO/2\OO$ is the field of 4 elements,
with unit group of order 3 on which the generator of $\Gal(F/\Que)$
acts by inversion.
For $\Delta(\OO)\ne3$, this makes $\Gal(H_{2,\OO}/L)$ into the symmetric group on 3 elements,
so the normal extension $L\subset L_2$ coincides with $L\subset H_{2,\OO}$, 
and we have $F\subset L_2\subset K_2$.
For the special value $\Delta=-3$, we have $H_{2,\OO}=H_\OO$ as
the group $(\OO/2\OO)^*/\OO^*$ is trivial, but from a Weierstrass equation 
$Y^2=X^3-a$ for $E$ we see that we still have
$F=\Que(\sqrt{-3})\subset L_2=\Split_L(X^3-a)\subset K_2$.
This proves (3).

We finally suppose that we have $\Delta(\OO)\equiv 0\bmod 4$ and
$F\not\subset K$, with $\Delta(\OO)\ne-4$.
In this case the ring $\OO/2\OO\iso \F_2[X]/(X^2)$ has a nilpotent 
maximal ideal $x\OO$ of index~2, and its unit group $\{1, 1+x\}$ has order 2.
As $L\subset H_{2,\OO}$ is a Galois extension of degree~4 that is the 
compositum of the quadratic extension $L\subset LF$ and the Galois extension 
$L\subset L_2$ of degree dividing 6, we conclude that $\Gal(H_{2,\OO}/L)$
is a Klein 4-group, and that $L\subset L_2$ and $L\subset LF$
are $L$-linearly disjoint quadratic extensions.
As we have $F\not\subset K$ by assumption, the quadratic extension $L_2\subset L_2F$ remains quadratic after taking composita with $K$, so $K_2$ does not contain $F$. 
In the excluded case $\Delta(\OO)=-4$, we can represent $E$ by a Weierstrass equation $Y^2=X^3-aX$ with $a\in K^*$.
In this case, $K\subset K_2=K(\sqrt{-a})$ is an extension of 
degree at most 2 that coincides with $F=\Que(i)$ if and only if we have $a\in {K^*}^2$. 
\end{proof}

\begin{proof}[Proof of Theorem \ref{thm:density not containing the CM field}]
We treat the cyclic reduction density of the sets of primes of $K$ for which 
$E$ has supersingular and ordinary reduction separately, 
and write $\delta_{E/K}$ accordingly as $\delta_{E/K}=\delta^\mathrm{ss}+\delta^\mathrm{ord}$

The supersingular primes are inert in the quadratic extension $K\subset KF$, so, by
Proposition \ref{prop:containment F in odd-division fields},
they do not split completely in the $\ell$-division fields $K_\ell$
when $\ell$ is an odd prime.
If $\Delta(\OO)$ is odd, the same applies to $\ell=2$.
In view of Corollary \ref{S_EKdescription}, we deduce that 
all primes $\gothp\nmid\Delta_K$ of supersingular reduction have cyclic reduction if $\Delta(\OO)$ is odd, leading to the contribution
$\delta^\mathrm{ss} = 1/2$ to $\delta_{E/K}$.
In the subcase $\Delta(\OO)\equiv 1\bmod 8$, we have $KF=K_2$, so the ordinary primes, which split completely in $K\subset KF$, cannot have cyclic reduction as the reduced curve will have complete 2-torsion.
Thus $\delta^\mathrm{ord}=0$ in this case.

When $\Delta(\OO)$ is even, supersingular primes are
primes of cyclic reduction if and only if they do not split 
completely in the extension $K\subset K_2$.
We find $\delta^\mathrm{ss} = \delta^\mathrm{ord} =0$ if we have $K=K_2$.
In all other cases $K\subset K_2$ and $K\subset KF$ are both quadratic extensions, and we find $\delta^\mathrm{ss} \in\{\frac{1}{4}, \frac{1}{2}\}$
depending on whether they are different or the same.
If they are the same, we have $\delta^\mathrm{ord} = 0$, since
the ordinary primes then split completely in $K_2=KF$.

The primes $\gothp$ of ordinary reduction of $E/K$ are split in the quadratic extension $K\subset KF$, and the point group $E(k_\gothp)$
is the same for these extension primes, which both have norm $N_{K/\Que}(\gothp)$.
Thus every ordinary prime $\gothp$ of cyclic reduction in $K$ corresponds to 
two ordinary primes of cyclic reduction in $KF$ of the same norm.
This yields $\delta^\mathrm{ord}=\frac{1}{2} \delta_{E/KF}$, and finishes the proof.

Note that as $KF$ does contain the CM-field $F$ of $E$, the ordinary density 
$\frac{1}{2}\delta_{E/KF}$ 
can be written in the form given by Theorem \ref{thm:density containing the CM field}.
\end{proof}
\noindent 
We summarize what we proved above Table 4.3, which gives more precise information than Theorem \ref{thm:density not containing the CM field} for the case $F\not\subset K$ and reduces the computation of $\delta_{E/K}$ in {\bf Case 3} to the computation of $\delta_{E/KF}$ in {\bf Case 2}.

\vspace{.5cm}
\centerline{{\bf 4.3. Table.} 
{\sl Supersingular and ordinary densities for}
$F\not\subset K$}
\vspace{-.3cm}

\begin{center}
\begin{table}[h]
\begin{tabular}{|c|c|c|ccc|}
\hline \rule{0pt}{3ex} \rule[-1.5ex]{0pt}{0pt}
\hspace{0.1cm} $\Delta(\OO)$   \hspace{0.1cm}      & \hspace{0.1cm} $1\bmod 8$ \hspace{0.1cm}   & \hspace{0.2cm} $5 \bmod 8$ \hspace{0.2cm}                & \multicolumn{3}{c|}{$0 \bmod 4$}                                                            \\ \hline \rule{0pt}{3ex}
$\delta^\mathrm{ss}$  & $\frac{1}{2}$ & $\frac{1}{2}$               & \multicolumn{1}{c|}{\hspace{0.07cm} $0$ \hspace{0.2cm}} & \multicolumn{1}{c|}{\hspace{0.07cm} $\frac{1}{4}$ \hspace{0.2cm}}               & \hspace{0.07cm} $\frac{1}{2}$ \hspace{0.1cm} \\ \rule{0pt}{2.8ex} \rule[-1.7ex]{0pt}{0pt}
$\delta^\mathrm{ord}$ & $0$           & $\frac{1}{2} \delta_{E/KF}$ & \multicolumn{1}{c|}{\hspace{0.07cm} $0$ \hspace{0.2cm}} & \multicolumn{1}{c|}{\hspace{0.07cm} $\frac{1}{2} \delta_{E/KF}$ \hspace{0.2cm}} & \hspace{0.07cm} $0$  \hspace{0.1cm}         \\ \hline 
\end{tabular}
\end{table}
\end{center}
\vspace{-.5cm}

\noindent
The three possibilities for $\delta^\mathrm{ss}\in\{0, \frac{1}{4}, \frac{1}{2}\}$
listed in the case $\Delta(\mathcal{O}) \equiv 0 \bmod 4$ correspond to the respective cases 
\[
K=K_2, \qquad \hbox{both extensions $K \subsetneq K_2 \subsetneq K_2 F$ quadratic},
\qquad K_2=KF.
\] 
Note that for the values $\Delta(\mathcal{O}) \equiv 0 \bmod 4$ different from 
$-4$, twisting the elliptic curve $E/K$ does not change $F$ and $K_2$ and which of the
three subcases we are in.
In the special case $\Delta(\mathcal{O})=-4$ we have $F=\Que(i)$, and $E$ is a twist of the curve with Weierstrass model $Y^2=X^3-aX$ over $K$ that has $K_2=K(\sqrt a)$.
Here all three possibilities mentioned above occur for certain $a\in K^*$.

\section{Vanishing of the density}

\noindent
The most common cause for the vanishing of the density 
$\delta_{E/K}$ in \eqref{deltaEK} is that we have $K=K_\ell$ for some prime $\ell$.
This can only occur for primes $\ell|2\Delta_K$, as $K=K_\ell$ implies that we have $\zeta_\ell\in K$.
In our {\bf Case 3} of CM-curves with CM-field 
$F\not \subset K$, Table 4.3 shows that this trivial cause for vanishing is the only possible one, and that it only happens for $\ell=2$.

If we have $K\ne K_\ell$ for all primes $\ell$ in either 
the non-CM {\bf Case 1} or the {\bf Case 2} of CM by $\OO\subset K$, there is 
still the interesting possibility of $\delta_{E/K}$  vanishing \emph{non-trivially}.
If one writes $\delta_{E/K}$ as explained in the Introduction as
\begin{equation}
\label{delta=alphaA_EK}
\delta_{E/K}= \alpha_{E/K}\cdot A_{E/K},
\end{equation}
with $A_{E/K} =\prod_{\ell \ \text{prime}} (1-[K_\ell:K]^{-1})$
the \emph{naive density} from \eqref{naive density}
and
$\alpha_{E/K}\in\Q_{\ge0}$ a rational \emph{entanglement correction factor},
then non-trivial vanishing amounts to having
\begin{equation}
\label{nontrivialvanishing}
\delta_{E/K}=\alpha_{E/K}=0\quad\text{and}\quad A_{E/K}>0.
\end{equation}
In this case all $K_\ell$ are different from $K$, but the
non-splitting conditions in the various $K_\ell$ cannot be
satisfied simultaneously.

Murty proved \cite{Murty3}*{Theorem 1} that non-trivial vanishing does not
happen for $K=\Q$: we have $\delta_{E/\Q}=0$ if and only if $E$ has full
2-torsion over $\Q$.
Over other number fields, non-trivial vanishing of $\delta_{E/K}$ is a rare occurrence, but we can
make it happen by base changing elliptic curves $E$ defined
over a small field such as $\Q$ to a well-chosen number field.

The underlying idea is an elliptic analogue of a similar
construction in the multiplicative setting \cite{Jones}*{Section 3}.
One starts with an elliptic curve $E/K$ with $\mathrm{End}_K(E)=\mathrm{End}_{\overline{K}}(E)$, i.e.,
in {\bf Case 1} or {\bf Case 2}, and considers $E$
over an extension $K\subset K'$ for which the
$\ell$-division fields $K'_\ell$ of $E$ over $K'$
for primes $\ell_1, \ell_2$, and $\ell_3$ are different
\emph{quadratic} extensions of $K'$, but with compositum $K'_{\ell_1 \ell_2 \ell_3}$
a multi-quadratic extension of $K'$ of degree 4, and not 8.
In this case, no prime of $K'$
can be inert in all three subextensions $K'_\ell$, and almost all reduced curves
at primes of $K'$ will have complete $\ell$-torsion for at least one value
$\ell\in\{\ell_1, \ell_2, \ell_3\}$, implying that $S_{E/K'}$ is finite.
The construction has many degrees of freedom. 
Both in {\bf Case 1} and {\bf Case 2}, it leads to
infinitely many different curves and number fields
for which non-trivial vanishing as in \eqref{nontrivialvanishing} occurs.

\begin{thm}
\label{vanishing theorem}
Let $E/K$ be any elliptic curve with 
$\mathrm{End}_K(E)=\mathrm{End}_{\overline{K}}(E)$, and
suppose that $E$ has naive density $A_{E/K}>0$.
Then for any finite normal extension $K\subset M$,
there exists a normal extension $K\subset K'$ that is $K$-linearly disjoint 
from $M$ and for which $\delta_{E/K'}$ vanishes non-trivially.
\end{thm}

\begin{proof}
Let $N$ be an integer divisible by all critical primes,
as in \eqref{NEK}.
Then the $N$-division field $K_N$ and the $\ell$-division fields
$K_\ell$ at $\ell\nmid N$ of $E$ form 
a linearly disjoint family over~$K$,
and $\Gal(K_\ell/K)$ is isomorphic to the full group
$\GL_2(\F_\ell)$ or $(\OO/\ell\OO)^*$, depending on whether 
$\mathrm{End}_{\overline{K}}(E)$ equals $\Zee$ or $\OO$.

Now let $K\subset M$ be a finite normal extension, and
replace $K_N$ by the compositum $MK_N$.
Then the family may no longer be $K$-linearly disjoint,
but it becomes $K$-linearly disjoint again after leaving out
finitely many $K_\ell$ from the family.
This is because any finite extension of $K$ contained in the
compositum of some $K$-linearly disjoint family of division fields
$K_\ell$ is contained in the compositum of \emph{finitely many}
$K_\ell$, and these are the ones that we leave out.

Now pick \emph{any} set $\{\ell_1,\ell_2,\ell_3\}$ of odd primes that have
not been left out.
Then the $\ell_1\ell_2\ell_3$-division field
$K_{\ell_1\ell_2\ell_3}$ of $E$ is Galois over $K$ with group
$$
G=
\prod_{i=1}^3 \mathrm{Gal}(K_{\ell_i}/K)=\prod_{i=1}^3 G_i
$$
where each $G_i=\mathrm{Gal}(K_{\ell_i}/K)$ is isomorphic to either $\GL_2(\F_{\ell_i})$ or to $(\OO/\ell_i\OO)^*$.
\noindent
Every $G_i$ thus contains a central subgroup
$\langle-1\rangle$ generated by $-\id_{\ell_i}\in \GL_2(\F_{\ell_i})$ or by $-1\in (\OO/\ell_i\OO)^*$,
so the center of $G$ contains an elementary abelian 2-group
$H'=\prod_{i=1}^3 \langle-1\rangle \subset G$ of order 8.
We let $H\subset H'$ be the `norm-1-subgroup' of order 4
consisting of elements $(e_i)_{i=1}^3\in H$ with $e_1e_2e_3=1$.
Then $H$ is normal in $G$,
and we take for $K'$ the invariant field $K'=K_{\ell_1\ell_2\ell_3}^H$.

We now view $E$ as an elliptic curve over the finite normal extension
$K'$ of $K$, and note that
the division field $K'_{\ell_1\ell_2\ell_3}=K_{\ell_1\ell_2\ell_3}$
is by construction Galois over $K'$ with group isomorphic to the
Klein four-group $H$.
As every non-trivial element of $H$ is the identity on exactly one
of the division fields $K'_{\ell_i}$,
the three intermediate quadratic extensions of
$K'\subset K'_{\ell_1\ell_2\ell_3}$ are the
division fields $K'_{\ell_i}$, and no prime of $K'$ will be inert in
all three of them.
This implies that we have $\delta_{E/K'}=0$.

As the naive density $A_{E/K'}$ in \eqref{naive density} differs
from $A_{E/K}>0$ only in the three factors corresponding to the primes
$\ell_i$, with the degree $[K_{\ell_i}:K]=\# G_i$
being replaced by $[K'_{\ell_i}:K']=2$, we still have $A_{E/K'}>0$,
so the vanishing of $\delta_{E/K'}$ is non-trivial.
\end{proof}

\noindent
\begin{remark}
Our proof of Theorem \ref{vanishing theorem} only uses the fact
that an element of order 2 is contained in $\Gal(K_{\ell_i}/K)$,
and that the Klein four-group $H$ in the proof is contained in
$G=\Gal(K_{\ell_1\ell_2\ell_3}/K)\subset \prod_{i=1}^3 G_i$.
This observation is useful when constructing an explicit example of an
elliptic curve $E$ over a `small' normal number field $K'$ for
which $\delta_{E/K'}$ vanishes non-trivially.
In the non-CM case, if one does not insist on $K'$ being normal over $K=\Q$,
one can use any element of order 2 in
$\Gal(K_{\ell_i}/K) \subset\GL_2(\F_{\ell_i})$ instead of $-\id_{\ell_i}$. In all cases, 
one can use small primes $\ell_i$ for which $K_{\ell_i}$ is of small even
degree.
\end{remark}

\begin{example}
\label{ex:vanishing}
The elliptic curve $E$ defined over $K=\Q$ by the minimal Weierstrass equation
$$y^2+xy+y=x^3-76x+298$$
does not have CM, has discriminant $\Delta_E=-2^{5} \cdot 5^{8}$
and, according to LMFDB \cite{LMFDB}, its mod $\ell$ Galois representations
are maximal at $\ell\ne 3, 5$.
The division field $K_3$ is non-abelian of degree 6,
smaller than the generic degree $48=\#\GL_2(\F_3)$,
and $K_5$ is Galois with group $A_{20}=C_5\rtimes \Aut(C_5)$,
the affine group over $\F_5$ of order 20, much smaller than
the generic group $\GL_2(\F_5)$ of order 480.
As the division fields $K_2$ and $K_3$ are non-abelian of degree~6
with different quadratic subfields $\Q(\sqrt{\Delta_E})=\Q(\sqrt{-2})$
and $\Q(\zeta_3)$, they are linearly disjoint over $K=\Q$.
As $K_6$ and $K_5$ are solvable extensions of $\Q$
with maximal abelian subfields $\Q(\sqrt{-2},\zeta_3)$ and $\Q(\zeta_5)$
that are linearly disjoint over $\Q$, the division fields
$K_2$, $K_3$ and $K_5$ are $\Q$-linearly disjoint of even degree,
so $\Gal(K_{30}/\Q)\iso S_3\times S_3\times A_{20}$ does contain an
elementary abelian 2-subgroup $H'$ of order 8 as in the proof of Theorem
\ref{vanishing theorem}.
We can therefore find a non-normal field $K'$ of degree
$[K':\Q]=6\cdot6\cdot 20/4=180$ inside $K_{30}$ for which
$\delta_{E/K'}$ vanishes non-trivially.
\end{example}

\begin{example} 
\label{ex:cm_vanishing}
We can find a `smaller' example of non-trivial vanishing if we turn our attention to CM elliptic curves. 
Put $\pi=1+\sqrt{-2}\in K=\Que(\sqrt{-2})$ and
consider the curve
\[
E_0: y^2+(\pi-1)xy + y = x^3+\pi x^2-\pi x
\]
with $j_{E_0}=8000$ over $K$. 
Then $E_0$ has CM by $\mathcal{O}=\mathbf{Z}[\sqrt{-2}]$, 
and the 3-division field $K_3=K(\sqrt{-3})$ is equal to the ray class field of conductor $3$ of $K$ generated by the $x$-coordinates of the points of order 3 in $E(\Kbar)$.
This implies that the family of $\ell$-division fields $K_\ell$ is linearly disjoint over $K$,
with $\mathrm{Gal}(K_\ell/K) \iso \left(\mathcal{O}/\ell \mathcal{O} \right)^\ast$ for all primes $\ell \neq 3$. 
In particular, $K_2$ and $K_5$ are cyclic extensions of $K$ of degree $2$ and $24$. 
This yields $\Gal(K_{30}/K)\iso C_2 \times C_2 \times C_{24}$ and, as in the previous example, we can find a field $K'$ of degree
$[K':\Q]=48$ inside $K_{30}$ for which
$\delta_{E_0/K'}$ vanishes non-trivially.

The curve $E_0$ has been obtained by twisting the elliptic curve $E$ from Example \ref{example_cm} in order to obtain `minimal Galois group' $\Gal(K_3/K)\iso (\OO/3\OO)^*/\{\pm1\}$ for the 3-division field. 
More precisely, we have $3=\pi\overline\pi$ and
\[
E[\overline\pi](\Kbar)= \textstyle 
\left 
\{O, 
\left 
( \frac{\pi}{3}, 
\pm \frac{2}{3}\sqrt{1-\frac{5}{3}\pi}
\right 
) 
\right \}.
\]
Hence, the twist of $E/K$ by 
$K(\sqrt{1-5\pi/3})=K(\sqrt{2+\sqrt{-2}})$ 
achieves the minimal $3$-division field by \cite{Campagna-Pengo}*{Theorem 5.11} and \cite{Campagna-Pengo2}*{Sections 4 and 5}. 
The curve $E_0$ given here is a minimal model for this twist.
\end{example}

\noindent
We do not know whether there exist examples of non-trivial
vanishing over number fields of degree less than 48.
We also do not know if there are examples that do not arise by
base change, i.e., an example in which
$\delta_{E/K}$ vanishes non-trivially for $K=\Q(j(E))$.
\section{Numerical examples}

\noindent
In this final Section, we will show by examples how one may
explicitly compute cyclic reduction densities.
The non-CM Case 1, with which we start, is the more complicated case, as already in the simplest case $K=\Que$, there are quite a few possibilities one may encounter for the non-generic Galois groups $\Gal(K_\ell/K)$. 
Moreover, as these typically non-abelian groups can be entangled in various ways, our small sample of examples can only sketch a very incomplete picture.

In order to compute the $\delta_{E/K}=\alpha_{E/K}\cdot A_{E/K}$
as in \eqref{delta=alphaA_EK} for a non-CM elliptic curve $E/K$, 
one needs explicit knowledge of finitely many Galois groups 
$\Gal(K_\ell/K)$ different from $\GL_2(\F_\ell)$.
For small examples over $\Que$, these are listed in LMFDB \cite{LMFDB}.
This list enables us to compute the naive density $A_{E/K}$ as a rational multiple of the elliptic non-CM Artin constant $A$ from \eqref{Ainfty}.

Finding the exact entanglement correction factor $\alpha_{E/K}$ 
for the fields $K_\ell$ with $\ell$ from the set $T_{E/K}$ of critical primes can be more complicated.
It typically involves group theory and ramification arguments.

Over $K=\Que$, there is a type of entanglement noticed by Serre that prevents the Galois representation on the torsion points of $E/\Que$ in all cases to be the full group $\GL_2(\Zhat)$.
It arises from the fact that the parity of a permutation of the three non-trivial 2-torsion points of $E$ under Galois is given by 2 \emph{different} quadratic characters on $\GL_2(\Zhat)$.
One is the parity map $\varepsilon: \GL_2(\Zhat)\tto\{\pm1\}$
that factors via $\GL(\F_2)\cong S_3$, the other is the 
composition $\chi_D: \GL_2(\Zhat)\ \mapright{\det}\ \Zhat^*\tto\{\pm1\}$
of the determinant map with the Dirichlet character of conductor 
$D=\disc({\Que(\sqrt{\Delta_E})})$ corresponding to the extension $\Que\subset \Que(\sqrt{\Delta_E})$.
As a consequence, the image of $G_\Que$ in $\GL_2(\Zhat)$ is contained in the kernel of the character $\varepsilon\chi_D$ of order 2.
Elliptic curves $E/\Que$ for which the image is equal to this subgroup of order 2 are known as \emph{Serre curves}.
Generically, elliptic curves over $\Que$ are known to be Serre curves \cite{Jones}, so this is by far the most common form of entanglement of division fields.

In the context of cyclic reduction, this Serre type entanglement
is only relevant if $D=\disc({\Que(\sqrt{\Delta_E})})$ 
is squarefree or, equivalently, congruent to $1\bmod 4$.
For Serre curves, the entanglement correction factor $\alpha_{E/\Que}$ is therefore
equal to 1 if $D$ is even,
and equal to 
\begin{equation}
\label{charsummethod}
\alpha_{E/\Q}=
1+\prod_{\ell\mid 2D, \ \ell \ \text{prime}} \frac{-1}{[K_\ell:\Q]-1}
\end{equation}
if $D$ is odd.
This follows from the \emph{character sum method}
\cite{LMS}*{Theorem 8.4},
which more or less mechanically produces entanglement correction
factors in all cases where, in terms of Definition \ref{def:linearlydisjoint}, the Galois group over $K$ of
the compositum $K_\infty$ of all $\ell$-division fields $K_\ell$ is a normal subgroup $\Gal(K_\infty/K)\subset \prod \Gal(K_\ell/K)$ for which the quotient is finite and abelian.
The simplest formula arises when the index equals 2.

For $K$ `small enough' to be numerically accessible, 
such as $K=\Que$,
the degrees $[K_\ell:K]$ going into the computation of $\alpha_{E/K}$
are typically close to their maximal degree, which is of order $\ell^4$.
The resulting value of $\alpha_{E/K}$ in \eqref{charsummethod}
is therefore often so close to 1 that the difference between 
the naive density $A_{E/K}$ 
and the actual cyclic reduction density $\delta_{E/K}$ 
cannot be observed numerically.
Clearly, only approximate correctness can be confirmed by a computer 
count of the fraction of primes of cyclic reduction among primes 
of norm below some modest bound.

For the non-CM curves $E$ over $K=\Que$ in Table 2 below, 
we computed the theoretical cyclic reduction density
$\delta_E$ and compared it to numerically computed
fraction $d_E(10^6)$ of the 78,498 rational primes
below $10^6$ for which the reduction was cyclic.
In all cases, the difference $|\delta_E-d_E(10^6)|$ was very small: at most $.0006$.

\vspace{.5cm}
\centerline{{\bf Table 2.} 
{\sl Seven examples of non-CM cyclic reduction densities over $K=\Que$}}
\vspace{-.2cm}

\begin{table}[h]
\begin{center}
\begin{adjustbox}{width=\columnwidth,center}
\begin{tabular}{|c|c|c|c|c|c|c|}
\hline \rule{0pt}{3ex} \rule[-1.5ex]{0pt}{0pt}
  $i$ & $E_i$ & $\Delta_{E_i}^{\text min}$  & $\alpha_{E_i}$ & $A_{E_i}$ & $\delta_{E_i}$ & $d_{E_i}(10^6)$ \\
 
\hline  \rule{0pt}{3ex} \rule[-1.5ex]{0pt}{0pt}
1 & $y^2+xy=x^3-x^2-x+1$ & $-2^2\cdot29$ & 1 & $A$ & .8138 & .8146\\

\hline \rule{0pt}{3ex} \rule[-1.5ex]{0pt}{0pt}
2 & $y^2=x^3+x+3$ & $-2^4\cdot13\cdot19$ & 1.0000 & $A$ & .8138 & .8141\\

\hline  \rule{0pt}{3ex} \rule[-1.5ex]{0pt}{0pt}
3 & $y^2+xy+y=x^3-76x+298$ & $-2^5\cdot5^8$ & 1 & $\frac{18240}{22513} A$ &  .6593 & .6588 \\

\hline \rule{0pt}{3ex} \rule[-1.5ex]{0pt}{0pt}
4 & $y^2=x^3-3x+1$ & $2^4\cdot3^4$ & 1 & $\frac{4}{5} A$ & .6510 & .6510\\

\hline \rule{0pt}{3ex} \rule[-1.5ex]{0pt}{0pt}
5 & $y^2=x^3+2x+3$ & $-2^4\cdot5^2\cdot11$ & 1.0001 & $\frac{3}{5} A$ & .4883 & .4889\\


\hline \rule{0pt}{3ex} \rule[-1.5ex]{0pt}{0pt}
6 & $y^2=x^3+x^2+7x$ & $2^2\cdot3\cdot7$ & 1.2 & $\frac{24}{47} A$ & .4986 &.4990\\

\hline \rule{0pt}{3ex} \rule[-1.5ex]{0pt}{0pt}
7 & $y^2+xy+y=x^3-36x-70$ & $-2^4\cdot3^3\cdot7^2$ & 1 & $\frac{72}{235} A$ & .2493 &.2490\\

\hline
\end{tabular}
\end{adjustbox}
\end{center}
\end{table}

\noindent
The curves in Table 2 are given by their minimal Weierstrass equations from LMFDB~\cite{LMFDB},
with discriminant $\Delta_E^{{\rm min}}$ dividing the discriminant $\Delta_E$
of a short Weierstrass model in~\eqref{Delta_E}.

\begin{example}
The curves $E_1$ and $E_2$ are more or less arbitrarily chosen Serre curves from LMFDB,
with $D_i=\disc(\Q(\sqrt{\Delta_{E_i}}))$ equal to $D_1=-2^2\cdot 29$ and  \hbox{$D_2=-13\cdot 19$}.
We omit the full proof that $E_1$ and $E_2$ are indeed Serre curves.
In the first case we have cyclic reduction density 
$\delta_{E_1/\Que}=A$, in the second case we have to multiply $A$ by an
entanglement correction factor 
$$
\alpha_{E_2}= 1+\prod_{\ell | 2\cdot 13\cdot 19} \frac{-1}{[K_\ell:\Q]-1}
        \approx 0.999999999938,
$$
that is numerically invisible in $\delta_{E_2}\approx A$.
\end{example}

\begin{example}
The curve $E_3$ with discriminant $\Delta_E=-2^{5} \cdot 5^{8}$
appeared in Example \ref{ex:vanishing}, 
where we proved that the family of division fields $K_\ell$ is linearly disjoint over $\mathbf{Q}$, 
with $[K_3:\mathbf{Q}]=6$ and $[K_5:\mathbf{Q}]=20$ having 
smaller degrees over $\Que$ than the 
generic values 48 and 480. 
Changing the corresponding Euler factors in $A$ yields a
naive density
\[
A_{E_3}=\prod_{\ell \ \text{prime}}
        \left(1-\frac{1}{[K_\ell:\Q]} \right)
        = \frac{5}{6} \cdot \frac{19}{20} \cdot \frac{48}{47} \cdot  \frac{480}{479} \cdot A = \frac{18240}{22513} \cdot A \approx 0.6455651
\]
that coincides with the actual density $\delta_{E_3}$ and is in good 
numerical agreement with $d_{E_3}$.
\end{example}

\begin{example}
For the elliptic curve $E_4: y^2=x^3-3x+1$
the splitting field $K_2$ of the
polynomial $x^3-3x+1$ is the real subfield $K_2=\Q(\zeta_9)^+$
of $\Q(\zeta_9)$, which is cubic, and not of maximal degree~6.
As all other $K_\ell$ have maximal degree, 
the naive density equals
\[
A_{E_4} = \frac{2}{3} \cdot \frac{6}{5} \cdot A =
\frac{4}{5} \cdot A\approx 0.6510015.
\]
The group $\Gal(K_3/\Q)\iso\GL_2(\F_3)$ has no quotient of order 3,
so the division field $K_6$ is a linearly disjoint compositum 
of $K_2$ and $K_3$.
We have $K_6\cap K_5=K$ as the intersection is solvable over $K$,
but does not contain $\sqrt 5\notin K_6$.
We can take $N=2\cdot3\cdot5$ in Theorem
\ref{teorema cruciale}, so we find that
the family of $\ell$-division fields $K_\ell$ is $\Q$-linearly disjoint, and that
$\delta_{E_4}$ equals the naive density $A_{E_4}$.
The numerical agreement is excellent.
\end{example}

\begin{example}
The elliptic curve $E_5: y^2=x^3+2x+3=(x+1)(x^2-x+3)$
has a unique rational torsion point of order 2 and 2-division field 
$K_2=\Q(\sqrt{-11})$.
For $\ell>2$, the degree of $K_\ell$ is maximal, so the naive density
equals
\[
A_{E_5} = \frac{1}{2} \cdot \frac{6}{5} \cdot A =
\frac{3}{5} \cdot A\approx 0.48825114.
\]
We can take $N=2\cdot3\cdot5\cdot11$ in Theorem \ref{teorema cruciale}.
As $K_2$ is not the unique quadratic subfield $\Q(\zeta_3)$ of $K_3$,
the extension $K_6$ is a linearly disjoint compositum of
$K_2$ and $K_3$.
Again, $K_6$ is solvable and does not contain $\sqrt 5$,
so it is linearly disjoint from $K_5$.

The family of division fields $\{K_\ell\}_{\ell\ne 11}$
is linearly disjoint over $\Q$, and 
$K_2=\Q(\sqrt{-11})\subset \Q(\zeta_{11})$
is a quadratic subfield of $K_{11}$: 
the familiar Serre-type entanglement.
In this case multiplying $A_{E_5}$ by $\alpha_{E_5}$ according 
to \eqref{charsummethod} amounts to leaving out the Euler factor 
$1-1/[K_{11}:\Q]=\frac{13199}{13200}$ from $A_{E_5}$:
any prime that does not split completely in $K_2$
\emph{automatically} does not do so in $K_{11}$, 
making the non-splitting condition in $K_{11}$ for
primes of cyclic reduction superfluous.
The resulting value
$\delta_{E_5}=\frac{13200}{13199}A_{E_5}\approx 0.4882881$
is too close to $A_{E_5}$ to detect 
the entanglement correction numerically.
\end{example}

\begin{example}
The entanglement correction can be made numerically visible by
selecting a curve such as $E_6: y^2=x(x^2+x+7)$, 
which has $K_2=\Que(\sqrt{-3})$ equal to the quadratic subfield 
of $K_3$ rather than of $K_{11}$, and a rational 3-torsion point $(1,3)$ that causes $K_3=K_6$ to be of degree 6.
There is no further entanglement, so the correction factor 
$\alpha_{E_6}=\frac{6}{5}= 1.2$ removing the Euler factor at 3
makes the naive density
$A_{E_6}=\frac{24}{47}A\approx .42$ very different from the
actual density $\Delta_{E_6}=\frac{144}{235}A\approx .50$.
\end{example}

\begin{example}
Our final example $E_7$ shows that for non-CM curves, small division
fields $K_\ell$ at small $\ell$ yield cyclic reduction densities 
that are very far from the `generic density' $A$ of about $.81$.
In the case of $E_7$, both
$K_2=\Que(\sqrt 2)$ and $K_3=\Que(\sqrt {-3})$ 
are quadratic and linearly disjoint from $K_5$ and $K_7$.
The cyclic reduction density 
$\delta_{E_7}=A_{E_7}=\frac{72}{235} A\approx .25$ is therefore
much smaller than $A$.
\end{example}

In the case of an elliptic curve $E$ with $CM$-order $\OO$, 
Theorem \ref{thm:density not containing the CM field} reduces
the calculation of $\delta_{E/K}$ to the case in which 
$K$ contains $\OO$, and therefore the 
ring class field $H_\OO$ of $\OO$.
The curve $E$ is a twist of a curve $E_0$ defined over
$H_\OO$, so we need to know how $\delta_{E_0/H_\OO}$
changes under twisting and base change to $K$.
As $\OO$ can be non-maximal, this is a somewhat involved story 
that we intend to address in more detail in future work.
We content ourselves with a second look Example \ref{example_cm}, which has $H_\OO=F$.

The elliptic curve $E: y^2=x^3+x^2-3x+1$ over $\Que$ 
from Example \ref{example_cm} has 
good reduction outside 2 and complex multiplication over $\overline\Q$ by the order $\OO=\Zee[\sqrt{-2}]$.
By Proposition \ref{prop:containment F in odd-division fields},
the CM-field $F=\Que(\sqrt{-2})$ is contained in all division fields 
$K_\ell$ for odd~$\ell$, but not in $K_2=\Que(\sqrt2)$.
Thus, $E(\F_p)$ has complete 2-torsion for $p\equiv\pm 1 \bmod 8$ 
and is cyclic for primes $p\equiv 5\mod 8$, 
which are inert in $K_2$ and in $F$.
Primes $p\equiv 3\mod 8$ yield two primes of norm $p$ in $F$
modulo which $E$ has point group $E(\F_p)$, so 
as in \eqref{prevdensity} we find
\[
    \delta_{E/\mathbf{Q}}=\textstyle \frac{1}{4} + \frac{1}{2}\delta_{E/F}.
\]
Over $F$, the Galois action on the torsion points of $E$ 
respects the $\OO$-module structure of 
$T=E^\tor(\overline\Que)\cong \lim_{\to n}\frac{1}{n}\OO/\OO$,
and CM-theory tells us that the associated Galois representation
\[ 
\rho_E: G_F\tto \Aut_\OO(T)=\widehat\OO^*
\]
realizes $\Gal(F(T)/F)$ as a subgroup of index at most 2 of
the unit group $\widehat\OO^*$ of the profinite completion 
$\widehat\OO=\lim_{\leftarrow n} \OO/n\OO$ of $\OO$.
On a finite level, the Galois group $\Gal(F_n/F)$ of the $n$-division field $F_n$ of $E$ over $F$ embeds into $(\OO/n\OO)^*$
as a subgroup of index 1 or 2. 
In our case, where $\OO$ has class number 1, the extension $F(T)$
generated by all torsion points of $E$ is abelian over $F=H_\OO$,
and therefore equal to the maximal abelian extension $F^\ab$ of $F$.
This implies that $\rho_E$ embeds 
$\Gal(F(T)/F)$ as an index-2 subgroup of $\widehat\OO^*$, which can also be described as the kernel of a continuous quadratic character
\[
\chi: \widehat{\mathcal{O}}^\ast = \prod_\ell \mathcal{O}_\ell \to \{\pm 1\} 
\]
where $\OO_\ell=\lim_{\leftarrow n} \OO/\ell^n\OO$ is the $\ell$-adic completion of $\mathcal{O}$.
For our curve $E$, the $\ell$-components of $\chi$ are all trivial if $\ell$ is odd 
(see \cite{Campagna-Pengo}*{Theorem 6.3}),
so the division fields $F_\ell$ for primes $\ell$ are linearly disjoint 
with maximal Galois groups $(\OO/\ell\OO)^*$ over $F$. 
For the cyclic reduction density over $F$, this means that we simply have
\[
\delta_{E/F}= A_\OO \approx .3403128,
\]
a figure that is nicely matched by $d_{E/F}(10^6)=.3416$.
The same holds for the densities
$\delta_{E/\Que}=\frac{1}{4}+\frac{1}{2} \delta_{E/F}$
and $\delta_{E/\Que(i)}=\frac{1}{2}$, as the fraction of primes 
$p\equiv 5\bmod8$ up to $10^6$ equals .2499, close to the limit value $\frac{1}{4}$.

In Example \ref{example_cm} we wrote $3=\pi\overline{\pi}$
with $\pi=1+\sqrt{-2}$, and twisted $E$ by $F(\sqrt{\pi+1})$
to obtain a curve 
$E_\pi: y^2+(\pi-1)xy + y = x^3+\pi x^2-\pi x$ 
for which the $\overline{\pi}$-division field $F_{\overline{\pi}}$
is not quadratic over $F$ but equal to it, and the
3-division field $F_3=F(\sqrt{-3})$ `minimal', with group 
$(\OO/3\OO)^*/\{\pm1\}$.
The family of $\ell$-division fields of $E_0$ is thus linearly disjoint over $K$ and we have
\[
\delta_{E_0/F}=\frac{1}{2} \cdot \frac{4}{3} \cdot A_\mathcal{O} \approx 0.2268752,
\]
in close agreement with the numerical value $d_{E_0/F}=.2269$ 
that we found.

\bibliographystyle{amsplain}

\end{document}